\documentclass[12pt,a4paper]{article}
\usepackage{ucs}
\usepackage[applemac]{inputenc}
\usepackage[T1]{fontenc}
\usepackage[english]{babel}		
\usepackage{
	amsmath,
	amsthm,
	amssymb,
	amsfonts,
	mathtools
}
\usepackage[dvips]{hyperref}
\hypersetup{	
	colorlinks,
	linkcolor=black,
	citecolor=black,
	urlcolor=black,
	breaklinks=true
}
\DeclareMathAlphabet{\mathpzc}{OT1}{pzc}{m}{it}
\usepackage{geometry}
\newtheoremstyle{lemma}{\topsep}{\topsep}
	{\itshape}
	{}
	{\bfseries}
	{.}
	{\newline}
	{\thmname{#1}\thmnumber{ #2}\thmnote{ #3}}	
\theoremstyle{lemma}
	\newtheorem{lemma}{Lemma}[section]

\newtheoremstyle{definition}{\topsep}{\topsep}
	{}
	{}
	{\bfseries}
	{.}
	{\newline}
	{\thmname{#1}\thmnumber{ #2}\thmnote{ #3}}	
\theoremstyle{definition}
	\newtheorem{definition}[lemma]{Definition}

	\newtheorem{warning}[lemma]{Warning}

\newcommand{\he}{\ensuremath{\alpha}}

\newcommand{\N}{\ensuremath{\mathbb{N}}}

\newcommand{\R}{\ensuremath{\mathbb{R}}}

\newcommand{\HM}{\ensuremath{\mathcal{H}}}

\newcommand{\B}{\ensuremath{\mathcal{B}}}

\newcommand{\V}{\ensuremath{\mathcal{V}}}

\newcommand{\M}{\ensuremath{\mathcal{M}}}

\DeclarePairedDelimiter\abs{\lvert}{\rvert}

\newcommand{\dd}{\ensuremath{\,\text{d}}}

\newcommand{\dHM}{\ensuremath{\,\text{d}\HM^{\he}}}
\newcommand{\dHMs}{\ensuremath{\,\text{d}\HM^{s}}}

\DeclareMathOperator{\dist}{dist}
\DeclareMathOperator{\diam}{diam}

\renewcommand{\phi}{\varphi}
\renewcommand{\epsilon}{\varepsilon}

\hypersetup{
	pdftitle={For which positive p is the integral Menger curvature Mp finite for all simple polygons?},
	pdfsubject={},
	pdfauthor={Sebastian Scholtes},
	pdfkeywords={},
	pdfcreator={},
	pdfproducer={}
}

\makeindex
\begin{document}

\title{For which positive $p$ is the integral Menger curvature $\M_{p}$ finite for all simple polygons?}
\author{\href{mailto:sebastian.scholtes@rwth-aachen.de}{Sebastian Scholtes}}
\date{November 24, 2011}
\maketitle

\begin{abstract}
	In this brief note\footnote{which is designated to be an addendum to \cite{Scholtes2011d}}
	we show that the integral Menger curvature $\M_{p}$ is finite for all simple polygons 
	if and only if $p\in (0,3)$. For the intermediate energies $\mathcal{I}_{p}$ and $\mathcal{U}_{p}$ we obtain
	the analogous result for $p\in (0,2)$ and $p\in (0,1)$, respectively.
\end{abstract}
\centerline{\small Mathematics Subject Classification (2000): 28A75; 53A04}
\bigskip

It is well known, and in fact, by finding similar triangles, pretty easy to prove, that any simple polygon that is not a straight line 
has infinite integral Menger curvature $\M_{p}$ for $p\geq 3$, cf.
\cite[Example after Lemma 1]{Strzelecki2007a} and \cite[after Theorem 1.2]{Strzelecki2011b} for similar energies.
This note investigates the opposite question, namely:

\begin{center}
	\textbf{
		Is there a $p\in (0,\infty)$, such that all simple polygons have finite integral Menger curvature $\M_{p}$?
	}
\end{center}

The answer to this question is:

\begin{center}
	\textbf{
		Yes, for all $p\in (0,3)$.
	}
\end{center}

Here the integral Menger curvature $\M_{p}(X)$, $p\in (0,\infty)$
of a set $X\subset \R^{n}$ is defined by
\begin{align*}
	\M_{p}(X)\vcentcolon=\int_{X}\int_{X}\int_{X}\kappa^{p}(x,y,z)\dHM_{X}(x)\dHM_{X}(y)\dHM_{X}(z),
\end{align*}
where the integrand $\kappa$ is the mapping
\begin{align*}
	\kappa:X^{3}\to \R,\,(x,y,z)\mapsto
	\begin{cases}
		r^{-1}(x,y,z),& x\not=y\not=z\not=x,\\
		0,&\text{else},
	\end{cases}
\end{align*}
and $r(x,y,z)$ is the radius of the circumcircle of the three points $x,y$ and $z$ -- if the points
are on a straight line we set $r(x,y,z)=\infty$, so that in this case $\kappa(x,y,z)=0$.\\

In a similar manner we can define the energies
\begin{align*}
	\mathcal{I}_{p}(X)\vcentcolon=\int_{X}\int_{X}\kappa_{i}^{p}(x,y)\dHM_{X}(x)\dHM_{X}(y)\quad\text{and}\quad
	\mathcal{U}_{p}(X)\vcentcolon=\int_{X}\kappa_{G}^{p}(x)\dHM_{X}(x),
\end{align*}
where
\begin{align*}
	\kappa_{i}(x,y)=\sup_{z\in X}\kappa(x,y,z)\quad\text{and}\quad
	\kappa_{G}(x)=\sup_{y,z\in X}\kappa(x,y,z).
\end{align*}
We also answer the analogous question for the intermediate energies $\mathcal{I}_{p}$ and $\mathcal{U}_{p}$, where
the appropriate parameter range is $p\in (0,2)$ and $p\in (0,1)$, respectively. To prove our result we show that
it is enough to control the energy of all polygons $E_{\phi}$ with two edges of length $1$ and angle $\phi\in (0,2\pi)$ and that 
these energies are controlled by the energy of $E_{\pi/2}$.\\

\textbf{Acknowledgement}\\
	The author wishes to thank H. von der Mosel for asking about this problem, reading the present note and improving it by 
	making several suggestions.

\begin{definition}[(The set $E_{\phi}$)]
	For $\phi\in \R$ we define
	\begin{align*}
		E_{\phi}\vcentcolon=[[0,1)\times\{0\}]\cup (\cos(\phi),\sin(\phi))[0,1).
	\end{align*}
\end{definition}

\begin{lemma}[(Estimate of $\kappa$ for $E_{\phi}$)]\label{estimateforkappaEphi}
	Let $\phi\in (0,2\pi)$. Then there is a constant $c(\phi)>0$, such that for all
	\begin{align*}
			x=(\xi,0), y=(\eta,0)\in (0,1]\times\{0\}\quad\text{and}\quad 
			z=\zeta(\cos(\phi),\sin(\phi))\in(\cos(\phi),\sin(\phi))(0,1].
	\end{align*}
	we have
	\begin{align*}
		\kappa(x,y,z)
			\leq c(\phi) \frac{2\zeta}{(\xi^{2}+\zeta^{2})^{1/2}(\eta^{2}+\zeta^{2})^{1/2}}.
	\end{align*}
\end{lemma}
\begin{proof}
		As $\kappa$ is invariant under isometries we only need to consider the case $\phi\in (0,\pi)$.
		We compute
		\begin{align*}
			\MoveEqLeft
			\kappa(x,y,z)=\frac{2\dist(z,L_{x,y})}{\abs{x-z}\abs{y-z}}\\
			&=\frac{2\sin(\phi)\zeta}{([\xi-\zeta\cos(\phi)]^{2}+[\zeta\sin(\phi)]^{2})^{1/2}([\eta-\zeta\cos(\phi)]^{2}+[\zeta\sin(\phi)]^{2})^{1/2}}\\
			&=\frac{2\sin(\phi)\zeta}{(\xi^{2}-2\xi\zeta\cos(\phi)+\zeta^{2})^{1/2}(\eta^{2}-2\eta\zeta\cos(\phi)+\zeta^{2})^{1/2}}.
		\end{align*}
		If $\phi\in [\pi/2,\pi)$ we have
		\begin{align*}
			\xi^{2}-2\xi\zeta\cos(\phi)+\zeta^{2}\geq \xi^{2}+\zeta^{2}
		\end{align*}
		and otherwise, i.e. $\phi\in (0,\pi/2)$
		\begin{align*}
			\MoveEqLeft
			\xi^{2}-2\xi\zeta\cos(\phi)+\zeta^{2}=[1-\cos(\phi)](\xi^{2}+\zeta^{2})
			+\underbrace{\cos(\phi)}_{\geq 0}\underbrace{[\xi^{2}-2\xi\zeta+\zeta^{2}]}_{\geq 0}\\
			&\geq [1-\cos(\phi)](\xi^{2}+\zeta^{2}).
		\end{align*}
\end{proof}

\begin{lemma}[(Estimate of $\mathcal{E}_{p}(E_{\phi})$ in terms of $\mathcal{E}_{p}(E_{\pi/2})$)]\label{inequalityEpi/2}
	Let $\phi\in\R$. Then there is a constant $c(\phi)>0$, such that for all $p\in (0,\infty)$, 
	$\mathcal{E}_{p}\in\{\mathcal{U}_{p},\mathcal{I}_{p},\mathcal{M}_{p}\}$ we have
	\begin{align*}
		\mathcal{E}_{p}(E_{\phi})\leq c(\phi)^{p} \mathcal{E}_{p}(E_{\pi/2}).
	\end{align*}
\end{lemma}
\begin{proof}
	Without loss of generality we might assume $\phi\in[0,2\pi]$ and as 
	$\mathcal{E}_{p}(E_{0})=\mathcal{E}_{p}(E_{2\pi})=\mathcal{E}_{p}(E_{\pi})=0$ for all $p\in (0,\infty)$
	we might as well assume $\phi\in (0,2\pi)\backslash\{\pi\}$. Let us denote
	\begin{align*}
		E_{\phi}^{1}\vcentcolon=(0,1)\times\{0\}\quad\text{and}\quad 
		E_{\phi}^{2}\vcentcolon=(\cos(\phi),\sin(\phi))(0,1).
	\end{align*}
	Define
	\begin{align*}
		f:E_{\phi}\to E_{\pi/2},\, x\mapsto
		\begin{cases}
			x,&x\in [0,1]\times\{0\},\\
			(0,x_{2}/\sin(\phi)),&x\in E_{\phi}^{2}.
		\end{cases}
	\end{align*}
	As $\kappa$ is invariant under isometries we can without loss of generality assume the situation of Lemma \ref{estimateforkappaEphi}
	and hence have
	\begin{align}\label{inequalitycurvatures}
		\kappa(x,y,z)\leq c(\phi)\kappa(f(x),f(y),f(z)),
	\end{align}
	if $\# \{x,y,z\in E_{\phi}^{1}\}\geq 1$ and $\# \{x,y,z\in E_{\phi}^{2}\}\geq 1$. Since
	$\kappa(x,y,z)=0$ for $x,y,z\in E_{\phi}^{1}\cup\{0\}$ or $x,y,z\in E_{\phi}^{2}\cup\{0\}$ we have (\ref{inequalitycurvatures}) 
	for all $x,y,z\in E_{\phi}$ and therefore by Lemma \ref{estimatechangeofvariablesmultiple}, 
	note that $f$ is bi-Lipschitz, proven the proposition.
\end{proof}

\begin{lemma}[(Range of $p$ where $\mathcal{E}_{p}(E_{\pi/2})$ is finite)]\label{finiteEpi/2}
	We have
	\begin{align*}
		\mathcal{U}_{p}(E_{\pi/2})&<\infty\quad\text{if and only if }\quad p\in (0,1),\\
		\mathcal{I}_{p}(E_{\pi/2})&<\infty\quad\text{if and only if }\quad p\in (0,2),\\
		\mathcal{M}_{p}(E_{\pi/2})&<\infty\quad\text{if and only if }\quad p\in (0,3).\\
	\end{align*}
\end{lemma}
\begin{proof}
	\cite[Theorem 1.1 and Proposition 1.2]{Scholtes2011d}
\end{proof}

\begin{lemma}[(Energy of polygons is determined by $E_{\phi}$)]\label{energyofpolygonsdeterminedbyEpi/2}
	Let $\phi\in\R$, fix $p\in (0,\infty)$ and $\mathcal{E}_{p}\in\{\mathcal{U}_{p},\mathcal{I}_{p},\mathcal{M}_{p}\}$, such that for all
	$\phi\in\R$ we have $\mathcal{E}_{p}(E_{\phi})<\infty$. Then if $P\subset\R^{n}$ is a simple polygon with finitely many vertices,
	we have $\mathcal{E}_{p}(P)<\infty$.
\end{lemma}
\begin{proof}
	Let $P\subset\R^{n}$ be a simple polygon with $N\geq 3$ vertices $x_{i}$, $i=1,\ldots,N-1$, 
	and denote by $\lambda>0$ the length of the shortest edge. 
	Then there is $\epsilon_{0}\in (0,\lambda/4)$, 
	such that for all $\epsilon\in (0,\epsilon_{0})$ the set $E_{i}\vcentcolon=P\cap B_{\epsilon}(x_{i})$ is some
	rescaled, rotated and translated version of a set $E_{\phi_{i}}$, because else the polygon would not be simple. 
	By $X_{i}$ we denote the edges of $P$ connecting $x_{i}$ and $x_{i+1}$. Then the $N-1$ sets 
	$Y_{i}\vcentcolon=X_{i}\backslash [E_{i}\cup E_{i+1}]$ are compact and $Y_{i}$ is
	disjoint to $Z_{i}\vcentcolon=\textrm{cl}(P\backslash X_{i})$, which is also compact. Therefore
	\begin{align*}
		d_{1}\vcentcolon=\min_{i=1,\ldots, N-1}\{\dist(Y_{i},Z_{i})\}/4>0,
	\end{align*}
	and for all $y\in Y_{i}$ we have
	\begin{align}\label{kappaformiddlesegment}
		\kappa(y,a,b)\leq d_{1}^{-1}\quad\text{if }a\in Z_{i}\text{ or }b\in Z_{i}.
	\end{align}
	As $P\backslash Z_{i}\subset X_{i}$, which is contained in a straight line, we even have (\ref{kappaformiddlesegment})
	for all $a,b\in P$. Now it remains to deal with the situation $y,a,b\not\in\bigcup_{i=1}^{N-1}Y_{i}$, since we can
	permute $y,a,b$ as arguments of $\kappa$ at will. This leads us to the two cases where either $y,a,b\in E_{i}$ or, 
	without loss of generality, $y\in E_{i}$ and $a\in E_{j}$ for $i\not= j$. If we denote
	\begin{align*}
		d_{2}\vcentcolon=\min_{\substack{i,j=1,\ldots,N-1\\i\not=j}}\{\dist(\textrm{cl}(E_{i}),\textrm{cl}(E_{j}))\}/4>0
	\end{align*}
	then the second case yields
	\begin{align*}
		\kappa(y,a,b)\leq d_{2}^{-1}
	\end{align*}
	and the first case is already controlled by Lemma \ref{inequalityEpi/2}, that is 
	$\mathcal{E}_{p}(E_{i})=\alpha_{i}\mathcal{E}_{p}(E_{\phi_{i}})$, where $\alpha_{i}\geq 0$ is the scaling constant.
	Now we can put all the cases together to estimate -- depending on which energy $\mathcal{E}_{p}$ we chose --
	\begin{align*}
		\MoveEqLeft
		\mathcal{U}_{p}(P)=\int_{\bigcup_{i=1}^{N-1}Y_{i}}\kappa_{G}^{p}(x)\dHM(x)
		+\int_{\bigcup_{i=1}^{N}E_{i}}\kappa_{G}^{p}(x)\dHM(x)\\
		&\leq \HM^{1}(P)d_{1}^{-p}+\int_{\bigcup_{i=1}^{N}E_{i}}\kappa_{G}^{p}(x)\dHM(x)
	\end{align*}
	with
	\begin{align*}
		\MoveEqLeft
		\int_{E_{i}}\kappa_{G}^{p}(x)\dHM(x)
		\leq \int_{E_{i}}\Big[\sup_{(y,z)\in \bigcup_{j=1}^{N-1} Y_{j}\times P}\kappa^{p}(x,y,z)\\
		&+\sup_{(y,z)\in \bigcup_{j\not=i} E_{j}\times P}\kappa^{p}(x,y,z)
		+\sup_{(y,z)\in  E_{i}\times E_{i}}\kappa^{p}(x,y,z)\Big]\dHM(x)\\
		\leq &\HM^{1}(P)(d_{1}^{-p}+d_{2}^{-p})+\mathcal{U}_{p}(E_{i})
		\leq \HM^{1}(P)(d_{1}^{-p}+d_{2}^{-p})+\alpha_{i}c(\phi_{i})^{p}\mathcal{U}_{p}(E_{\pi/2})<\infty
	\end{align*}
	or
	\begin{align*}
		\MoveEqLeft[1]
		\mathcal{I}_{p}(P)=2 \int_{P}\int_{\bigcup_{l=1}^{N}Y_{l}}\kappa_{i}^{p}(x,y)\dHM(x)\dHM(y)\\
		&+\sum_{l\not= k}\int_{E_{l}}\int_{E_{k}}\kappa_{i}^{p}(x,y,z)\dHM(x)\dHM(y)
		+\sum_{l=1}^{N}\int_{E_{l}}\int_{E_{l}}\kappa_{i}^{p}(x,y,z)\dHM(x)\dHM(y)\\
		\leq&  [\HM^{1}(P)]^{2}(2d_{1}^{-p}+N^{2}d_{2}^{-p})+\sum_{l=1}^{N}\int_{E_{l}}\int_{E_{l}}\kappa_{i}^{p}(x,y,z)\dHM(x)\dHM(y),
	\end{align*}
	with
	\begin{align*}
		\MoveEqLeft[1]
		\int_{E_{l}}\int_{E_{l}}\kappa_{i}^{p}(x,y,z)\dHM(x)\dHM(y)\leq
		\int_{E_{l}}\int_{E_{l}}\sup_{z\in \bigcup_{j}Y_{j}}\kappa^{p}(x,y,z)\dHM(x)\dHM(y)\\
		&+\int_{E_{l}}\int_{E_{l}}\sup_{z\in \bigcup_{j\not= l}E_{j}}\kappa^{p}(x,y,z)\dHM(x)\dHM(y)+
		\int_{E_{l}}\int_{E_{l}}\sup_{z\in E_{l}}\kappa^{p}(x,y,z)\dHM(x)\dHM(y)\\
		\leq & [\HM^{1}(P)]^{2}(d_{1}^{-p}+d_{2}^{-p})+\mathcal{I}_{p}(E_{l})\leq 
		[\HM^{1}(P)]^{2}(d_{1}^{-p}+d_{2}^{-p})+\alpha_{l}c(\phi_{l})^{p}\mathcal{I}_{p}(E_{\pi/2})<\infty
	\end{align*}
	or
	\begin{align*}
		\MoveEqLeft
		\M_{p}(P)=3 \int_{P}\int_{P}\int_{\bigcup_{i=1}^{N}Y_{i}}\kappa^{p}(x,y,z)\dHM(x)\dHM(y)\dHM(z)\\
		&+\sum_{\#\{i,j,k\}\geq 2}\int_{E_{i}}\int_{E_{j}}\int_{E_{k}}\kappa^{p}(x,y,z)\dHM(x)\dHM(y)\dHM(z)
		+\sum_{i=1}^{N}\alpha_{i}\mathcal{M}_{p}(E_{\phi_{i}})\\
		\leq&  [\HM^{1}(P)]^{3}(3d_{1}^{-p}+N^{3}d_{2}^{-p})+\Big(\sum_{i=1}^{N}\alpha_{i}c(\phi_{i})^{p}\Big)\mathcal{M}_{p}(E_{\pi/2})
		<\infty.
	\end{align*}
\end{proof}

By $\mathcal{P}\subset \mathrm{Pot}(\R^{n})$ we denote the set of all simple polygons with finitely many vertices.

\begin{lemma}[(Polygons have finite $\mathcal{U}_{p}$ iff $p\in (0,1)$)]\label{polygonsfiniteUp}
	Let $p\in (0,\infty)$. The following are equivalent
	\begin{itemize}
		\item
			$p\in (0,1),$
		\item
			$\mathcal{U}_{p}(P)<\infty$ for all $P\in \mathcal{P},$
		\item
			there is a non-degenerate closed polygon $P$, such that $\mathcal{U}_{p}(P)<\infty$.
	\end{itemize}
\end{lemma}
\begin{proof}
	This is clear by Lemma \ref{finiteEpi/2} and Lemma \ref{energyofpolygonsdeterminedbyEpi/2} together with \cite[Theorem 1.1]{Scholtes2011d} and the
	information that any vertex of a polygon with angle in $(0,2\pi)\backslash \{\pi\}$ has no approximate $1$-tangent at this vertex.
\end{proof}

\begin{lemma}[(Polygons have finite $\mathcal{I}_{p}$ iff $p\in (0,2)$)]
 	Let $p\in (0,\infty)$. The following are equivalent
	\begin{itemize}
		\item
			$p\in (0,2),$
		\item
			$\mathcal{I}_{p}(P)<\infty$ for all $P\in \mathcal{P},$
		\item
			there is a non-degenerate closed polygon $P$, such that $\mathcal{I}_{p}(P)<\infty$.
	\end{itemize}
\end{lemma}
\begin{proof}
	See the proof of Lemma \ref{polygonsfiniteUp}.
\end{proof}

\begin{lemma}[(Polygons have finite $\M_{p}$ iff $p\in (0,3)$)]
	Let $p\in (0,\infty)$. The following are equivalent
	\begin{itemize}
		\item
			$p\in (0,3),$
		\item
			$\mathcal{M}_{p}(P)<\infty$ for all $P\in \mathcal{P},$
		\item
			there is a non-degenerate closed polygon $P$, such that $\mathcal{M}_{p}(P)<\infty$.
	\end{itemize}
\end{lemma}
\begin{proof}
	See the proof of Lemma \ref{polygonsfiniteUp}.
\end{proof}

\begin{appendix}

\section{Appendix: Some remarks on integration}

In this section we give some remarks on how to get estimates for the change of variables formula.
Suppose we have a homeomorphism $g:X\to Y$ between two metric spaces and an integrand $f:X\cup Y\to\overline\R$ for which we know that
$f\leq f\circ g$ on $X$. Under which circumstances can we estimate in the following way
\begin{align*}
	\int_{X}f\dHMs_{X}\leq \int_{X}f\circ g\dHMs_{X}\leq C \int_{Y}f\dHMs_{Y}\quad ?
\end{align*}

\begin{lemma}[(Estimate for change of variables formula)]\label{estimatechangeofvariables}
	Let $(X,d_{X})$, $(Y,d_{Y})$ be metric spaces. Let $s\in (0,\infty)$, $f:Y\to \overline \R$ be $\B(Y)$--$\B(\overline \R)$ measurable, $f\geq 0$ and
	$g:X\to Y$ be a homeomorphism, with $d_{X}(g^{-1}(y_{1}),g^{-1}(y_{2}))\leq cd_{Y}(y_{1},y_{2})$ for all $y_{1},y_{2}\in Y$.
	Then
	\begin{align*}
		\int_{X}f\circ g\dHMs_{X}\leq c^{s}\int_{Y}f\dd \HM^{s}_{Y}.
	\end{align*}
\end{lemma}
\begin{proof}
	\textbf{Step 1}
		Let $V\subset Y$ and $(V_{n})_{n\in\N}$ be a $\delta$ covering of $V$. Then $U_{n}=g^{-1}(V_{n})$ 
		cover $U=g^{-1}(V)$ with
		\begin{align*}
			\diam(g^{-1}(V_{n}))\leq c\diam(V_{n})\leq c\delta.
		\end{align*}
		Consequently we have $g_{*}(\HM^{s}_{X})(V)=\HM^{s}_{X}(g^{-1}(V))\leq c^{s}\HM^{s}_{Y}(V)$.\\
	\textbf{Step 2}
		As $f\geq 0$ is Borel measurable, i.e. $\B(Y)$--$\B(\overline\R)$ measurable, Lemma \ref{approximationofmeasurablefunctions}
		gives us non-negative Borel measurable simple functions 
		$u_{n}:Y\to \overline\R$, $u_{n}\uparrow f$. According to the Monotone Convergence Theorem this gives us
		\begin{align*}
			\int_{Y}f\dd g_{*}(\HM^{s}_{X})=\lim_{n\to\infty}\int_{Y}u_{n}\dd g_{*}(\HM^{s}_{X})
			\leq \lim_{n\to\infty}\int_{Y}c^{s}u_{n}\dd\HM^{s}_{Y}=c^{s}\int_{Y}f\dd \HM^{s}_{Y}.
		\end{align*}
		The previous estimate and use of Monotone Convergence Theorem is only justified, because
		\begin{align*}
			\B(Y)\subset \mathcal{C}(\HM^{s}_{Y})\quad\text{and}\quad
			\B(Y)\subset g(\mathcal{C}(\HM^{s}_{X}))=\mathcal{C}(g_{*}(\HM^{s}_{X}))			
		\end{align*}
		by Lemma \ref{characterisationofmeasurablesets} together with the fact that $g$ is a
		homeomorphism and hence maps $\B(X)$ onto $\B(Y)$.\\
	\textbf{Step 3}
		Now we can use Lemma \ref{changeofvariables} to write
		\begin{align*}
			\int_{X}f\circ g\dHMs_{X}=\int_{Y}f\dd g_{*}(\HM^{s}_{X})\leq c^{s}\int_{Y}f\dd \HM^{s}_{Y}.
		\end{align*}
\end{proof}

\begin{lemma}[(Estimate for change of variables formula in multiple integrals)]\label{estimatechangeofvariablesmultiple}
	Let $(X,d_{X})$, $(Y,d_{Y})$ be metric spaces. Let $s\in (0,\infty)$, $f:Y^{n}\to \overline \R$ be lower semi-continuous, $f\geq 0$ and
	$g:X\to Y$ be a homeomorphism, with $d_{X}(g^{-1}(y_{1}),g^{-1}(y_{2}))\leq cd_{Y}(y_{1},y_{2})$ for all $y_{1},y_{2}\in Y$.
	Then
	\begin{align*}
		\MoveEqLeft
		\int_{X}\ldots\int_{X}f(g(x_{1}),\ldots,g(x_{n}))\dHMs_{X}(x_{1})\ldots \dHMs_{X}(x_{n})\\
		&\leq c^{sn}\int_{Y}\ldots\int_{Y}f(y_{1},\ldots,y_{n})\dd\HM^{s}_{Y}(y_{1})\ldots\dd\HM^{s}_{Y}(y_{n}).
	\end{align*}
\end{lemma}
\begin{proof}
	\textbf{Step 1}
		For fixed $v_{1},\ldots,v_{n}\in Y$ and $a_{k},a\in Y$ with $a_{n}\to a$ we have 
		\begin{align*}
			f(v_{1},\ldots,v_{l-1},a,v_{l+1},\ldots,v_{n})\leq \liminf_{k\to\infty}f(v_{1},\ldots,v_{l-1},a_{k},v_{l+1},\ldots,v_{n})
		\end{align*}
		and hence by Fatou's Lemma
		\begin{align*}
			\MoveEqLeft
			\int_{Y}f(y_{1},v_{2}\ldots,v_{l-1},a,v_{l+1},\ldots,v_{n})\dHMs(y_{1})\\ 
			&\leq\int_{Y}\liminf_{k\to\infty}f(y_{1},v_{2},\ldots,v_{l-1},a_{k},v_{l+1},\ldots,v_{n})\dHMs(y_{1})\\
			&\leq \liminf_{k\to\infty}\int_{Y}f(y_{1},v_{2},\ldots,v_{l-1},a_{k},v_{l+1},\ldots,v_{n})\dHMs(y_{1}),
		\end{align*}
		so that $y\mapsto \int_{Y}f(y_{1},v_{2},\ldots,v_{l-1},y,v_{l+1},\ldots,v_{n})\dHMs(y_{1})$ is lower semi-continuous.
		Hence
		\begin{align*}
			\MoveEqLeft
			\int_{Y}\int_{Y}f(y_{1},y_{2},v_{3},\ldots,v_{l-1},a,v_{l+1}\ldots,v_{n})\dHMs(y_{1})\dHMs(y_{2})\\
			&\leq \int_{Y}\liminf_{k\to\infty}\int_{Y}f(y_{1},y_{2},v_{3},\ldots,v_{l-1},a_{k},v_{l+1}\ldots,v_{n})\dHMs(y_{1})\dHMs(y_{2})\\
			&\leq \liminf_{k\to\infty}\int_{Y}\int_{Y}f(y_{1},y_{2},v_{3},\ldots,v_{l-1},a_{k},v_{l+1}\ldots,v_{n})\dHMs(y_{1})\dHMs(y_{2})
		\end{align*}
		and by a straightforward inductive argument we can show that for all $l\in\{2,\ldots,n\}$ the mappings
		\begin{align*}
			Y\to\overline\R,\, y\mapsto \int_{Y}\ldots\int_{Y}f(y_{1},\ldots,y_{l-1},y,v_{l+1},\ldots,v_{n})\dHMs(y_{1})\ldots\dHMs(y_{l-1})
		\end{align*}
		are lower semi-continuous for all $v_{1},\ldots,v_{n}\in Y$ and hence also $\B(Y)$--$\B(\overline\R)$ measurable.\\
	\textbf{Step 2}
		Now we can successively use Lemma \ref{estimatechangeofvariables} to obtain
		\begin{align*}
			\MoveEqLeft
			\int_{X}\ldots\int_{X}f(g(x_{1}),\ldots,g(x_{n}))\dHMs_{X}(x_{1})\ldots \dHMs_{X}(x_{n})\\
			&\leq \int_{X}\ldots\int_{X}c^{s}\int_{Y}f(y_{1},g(x_{2})\ldots,g(x_{n}))\dHMs_{Y}(y_{1})\dHMs_{X}(x_{2})\ldots \dHMs_{X}(x_{n})\\
			&\leq \ldots\leq c^{sn}\int_{Y}\ldots\int_{Y}f(y_{1},\ldots,y_{n})\dd\HM^{s}_{Y}(y_{1})\ldots\dd\HM^{s}_{Y}(y_{n}).
		\end{align*}
\end{proof}

\begin{warning}[(For L. \ref{estimatechangeofvariablesmultiple} the hypothesis $f$ Borel measurable is not enough)]
	For the argument used in the proof of Lemma \ref{estimatechangeofvariablesmultiple} it would not suffice to have $f:Y^{n}\to\overline\R$
	Borel measurable, because then we would not be able to show that $f(\cdot,v_{2},\ldots,v_{n}):Y\to\overline\R$ is Borel measurable
	-- as Suslin showed that there are Borel sets, whose projections are not Borel sets --
	which was a hypothesis of Lemma \ref{estimatechangeofvariables}.
\end{warning}

\begin{lemma}[(Approximation of measurable functions with simple functions)]\label{approximationofmeasurablefunctions}
	Let $(X,\mathcal{A})$ be a measurable space, $f:(X,\mathcal{A})\to (\overline \R,\B(\overline \R))$, $f\geq 0$. 
	Then $f$ is measurable if and only if there is a sequence of simple, non-negative, measurable functions
	$u_{n}:(X,\mathcal{A})\to (\overline \R,\B(\overline \R))$, with $u_{n}\uparrow f$.
\end{lemma}
\begin{proof}
	\cite[III \S 4 Satz 4.13, p.108]{Elstrodt2005a}
\end{proof}

\begin{lemma}[(Change of variables)]\label{changeofvariables}
	Let $\V$ be a Borel regular outer measure on $X$, $Y$ be a set and $g:X\to Y$ a 
	bijective map.
	Further let $f:(Y,\mathcal{C}(g_{*}(\V)))\to (\overline\R,\B(\overline\R))$ measurable, $f\geq 0$. Then
	\begin{align}\label{changeofvariablesformula}
		\int_{Y}f\dd g_{*}(\V)=\int_{X}f\circ g\dd \V.
	\end{align}
\end{lemma}
\begin{proof}
	As we have a setting that the reader might find to be slightly confusing, we will proof this lemma. It is essentially
	the proof that can be found in \cite[V \S 3 3.1, p.191]{Elstrodt2005a}.\\
	\textbf{Step 1}
		Let $h:(Y,\mathcal{C}(g_{*}(\V)))\to (\overline\R,\B(\overline\R))$ be measurable and $B\in\B(\overline\R)$. Then
		\begin{align*}
			(h\circ g)^{-1}(B)=g^{-1}(\underbrace{h^{-1}(B)}_{\mathclap{\in \mathcal{C}(g_{*}(\V))
			\stackrel{\text{L. \ref{characterisationofmeasurablesets}}}{=}g(\mathcal{C}(\V))}})
			\in \mathcal{C}(\V),
		\end{align*}	
		so that $h\circ g$ is $\mathcal{C}(\V)$--$\B(\overline\R)$ measurable.\\
	\textbf{Step 2}
		For all $E\in\mathcal{C}(g_{*}(\V))$,
		i.e. $g^{-1}(E)\in \mathcal{C}(\V)$ by Lemma \ref{characterisationofmeasurablesets}, we have
		\begin{align*}
			\int_{Y}\chi_{E}\dd g_{*}(\V)=\V(g^{-1}(E))=\int_{X}\chi_{g^{-1}(E)}\dd\V
			=\int_{X}\chi_{E}\circ g\dd\V,
		\end{align*}
		because $\chi_{E}\circ g$ is $\mathcal{C}(\V)$--$\B(\overline\R)$ measurable by Step 1.
		Consequently we have the change of variables formula (\ref{changeofvariablesformula}) with $u$ instead of $f$, 
		for all simple, non-negative, measurable functions $u:(X,\mathcal{C}(g_{*}(\V)))\to (\overline \R,\B(\overline \R))$.\\
	\textbf{Step 3}
		As $f\geq 0$ is $\mathcal{C}(g_{*}(\V))$--$\B(\overline\R)$ measurable we know from Lemma \ref{approximationofmeasurablefunctions},
		that there is a sequence of simple, non-negative, measurable functions	
		$u_{n}:(X,\mathcal{C}(g_{*}(\V)))\to (\overline \R,\B(\overline \R))$, with $u_{n}\uparrow f$. By the Monotone Convergence Theorem
		\cite[1.3, Theorem 2, p.20]{Evans1992a} together with Step 2 we obtain
		\begin{align*}
			\int_{Y}f\dd g_{*}(\V)=\lim_{n\to\infty}\int_{Y}u_{n}\dd g_{*}(\V)
			=\lim_{n\to\infty}\int_{X}u_{n}\circ g\dd\V
			=\int_{X}f\circ g\dd\V,
		\end{align*}
		as $u_{n}\circ g$ are simple, non-negative $\mathcal{C}(\V)$--$\B(\overline\R)$ measurable functions
		with $u_{n}\circ g\uparrow f\circ g$.
\end{proof}

\begin{lemma}[(What is $\mathcal{C}(g_{*}(\V))$?)]\label{characterisationofmeasurablesets}
	Let $\V$ be an outer measure on $X$, $Y$ be a set and $g:X\to Y$ a bijective map. Then
	\begin{align*}
		\mathcal{C}(g_{*}(\V))=g(\mathcal{C}(\V)).
	\end{align*}
\end{lemma}
\begin{proof}
	\textbf{Step 1}
		Let $E\in \mathcal{C}(g_{*}(\V))$ and $U\subset X$. Then
		\begin{align*}
			\MoveEqLeft
			\V(g^{-1}(E))=g_{*}(\V)(E)=g_{*}(\V)(E\cap g(U))+g_{*}(\V)(E\backslash g(U))\\
			&=\V(g^{-1}(E\cap g(U)))+\V(g^{-1}(E\backslash g(U)))
			=\V(g^{-1}(E)\cap U))+\V(g^{-1}(E)\backslash U)),
		\end{align*}
		so that $g^{-1}(E)\in \mathcal{C}(V)$ and hence $E\in g(\mathcal{C}(\V))$.\\
	\textbf{Step 2}
		Let $E\in g(\mathcal{C}(\V))$ and $V\subset Y$. Then
		\begin{align*}
			\MoveEqLeft
			g_{*}(\V)(E)=\V(g^{-1}(E))=\V(g^{-1}(E)\cap g^{-1}(V))+\V(g^{-1}(E)\backslash g^{-1}(V))\\
			&=\V(g^{-1}(E\cap V))+\V(g^{-1}(E\backslash V))=g_{*}(\V)(E\cap V)+g_{*}(\V)(E\backslash V),
		\end{align*}
		which gives us $E\in  \mathcal{C}(g_{*}(\V))$.
\end{proof}

\end{appendix}

\bibliography{smalllibrary.bib}{}
\bibliographystyle{amsalpha}
\bigskip
\noindent
\parbox[t]{.8\textwidth}{
Sebastian Scholtes\\
Institut f{\"u}r Mathematik\\
RWTH Aachen University\\
Templergraben 55\\
D--52062 Aachen, Germany\\
sebastian.scholtes@rwth-aachen.de}

\end{document}